\documentclass[11pt,a4paper,reqno]{amsart}

\usepackage{a4wide}

\usepackage[T1]{fontenc}
\usepackage[utf8]{inputenc}
\usepackage{amsmath}
\usepackage{amsfonts}
\usepackage{dsfont}
\usepackage{paralist} 
\usepackage{amsthm}
\usepackage[colorinlistoftodos]{todonotes}
\usepackage{txfonts}
\usepackage[all]{xy}
\usepackage{amssymb}
\usepackage{mathrsfs} 
\usepackage{changes}
\usepackage{url}
\title{Existence and unicity of co-moments in multisymplectic geometry}
\author{Leonid Ryvkin\textsuperscript{$\star$}}
\thanks{\textsuperscript{$\star$}The first author wants to thank the Gerhard C. Starck Stiftung for financial and moral support.}

\address{%
Fakult\"at f\"ur Mathematik\\
Ruhr-Universit\"at Bochum\\
44780 Bochum\\
\mbox{Germany}
}
\email{leonid.ryvkin@rub.de}

\author{Tilmann Wurzbacher}

\address{%
Institut \'Elie Cartan\\
Universit\'e de Lorraine et
C.N.R.S.\\
57045 Metz, France}
\curraddr{%
Fakult\"at f\"ur Mathematik\\
Ruhr-Universit\"at Bochum\\
44780 Bochum\\
\mbox{Germany}
}
\email{tilmann.wurzbacher@rub.de}

\date{November 5, 2014}

\theoremstyle{definition}
\newtheorem{Definition}{Definition}[section]
\newtheorem{Theorem}[Definition]{Theorem}
\newtheorem{Corollary}[Definition]{Corollary}
\newtheorem{Lemma}[Definition]{Lemma}
\newtheorem{Remark}[Definition]{Remark}
\newtheorem{Proposition}[Definition]{Proposition}

\begin{document}

\begin{abstract}
Given a multisymplectic manifold $(M,\omega)$ and a Lie algebra $\frak{g}$
acting on it by infinitesimal symmetries, Fregier-Rogers-Zambon
define a homotopy (co-)moment as an $L_{\infty}$-algebra-homomorphism
from $\frak{g}$ to the observable algebra $L(M,\omega)$ associated to $(M,\omega)$, in analogy with and generalizing 
the notion of a co-moment map in symplectic geometry. We give a cohomo\-logical characterization 
of existence and unicity for homotopy co-moment maps and show its utility in multisymplectic 
geometry by applying it to special cases as exact multisymplectic manifolds and simple Lie groups and
by deriving from it existence results concerning partial co-moment maps, as {\it e.g.} covariant 
multimomentum maps and multi-moment maps.  

\end{abstract}

\setcounter{section}{-1}
\maketitle
\section{Introduction}

A manifold $M$ together with a non-degenerate, closed $(n{+}1)$-form $\omega$ is called
$n$-plectic or multisymplectic. Interest in the geometry of these manifolds stems mainly form the 
quest for a multiphase space
or Hamiltonian formulation for classical field theory generalizing Hamiltonian mechanics ({\it cf. e.g.} 
\cite{MR1244450}, \cite{forger-paufler-roemer}, \cite{helein-kouneiher} and references therein).
Unfortunately the Poisson brackets on functions on a symplectic manifold do not easily generalize
to the field-theoretic situation. On the one hand expected identities hold on multiphase spaces
only up to ``divergence terms'' and on the other hand ``symplectic structures'' on 
infinite dimensional mapping spaces, the phase spaces of classical field theories, are typically highly
delicate to work with (compare {\it e.g.}  \cite{barnich-shLiefieldtheory}, \cite{mm-clft2} and \cite{Helein1106} for a 
mathematical discussion of aspects of these issues).\\

A crucial insight by Baez-Hoffnung-Rogers and Rogers ({\it cf.} \cite{MR2566161} resp. \cite{MR2892558}),
extending the cited work of Barnich-Fulp-Lada-Stasheff (\cite{barnich-shLiefieldtheory}) on 
algebraic structures in classical field theory,
is that every $n$-plectic manifold comes with a Lie $n$-algebra, a special case of an $L_{\infty}$-algebra,
of observables. This result explains the presence of divergence terms and allows
to study classical field theories via finite dimensional geometry with a reasonable analogue of
the Poisson bracket on the space of observables. The next important step in the development of multisymplectic 
geometry is the definition of a homotopy (co-)moment map and the elucidation of its fundamental 
properties in the preprint \cite{fregier2013homotopy}.\\

Since moment maps play an immense role in symplectic geometry, one is summoned to ask for its
existence and unicity in the multisymplectic situation. In the symplectic case this is, of course, well-studied 
and easily formulated in cohomological terms (compare {\it e.g.}  \cite{weinstein}, pp. 20-21). Our lack of understanding of the 
cohomological description in Section 4.2 in \cite{MR1244450} led us to look for an obstruction theory for 
homotopy co-moment maps that is at the same time concise and implies the partial results scattered
in the literature.\\

Let us describe the content of this note in some more detail. The first section recalls the work of Rogers 
on the observable algebra of an $n$-plectic manifold $M$ (with the extension to the pre-$n$-plectic case
given by Zambon), whereas the second section defines -following \cite{fregier2013homotopy}- the notion
of a homotopy (co-)moment map for a given infinitesimal $n$-plectic action 
$\zeta:\mathfrak g\to \mathfrak X(M,\omega)$, where $\mathfrak g$ is a finite-dimensional Lie algebra 
and $\mathfrak X(M,\omega)$ is the space of vector fields whose flow preserve the $n$-plectic form $\omega$.
Section 3 first sets up the double-complex $\Lambda^\bullet\mathfrak g^*\otimes \Omega^\bullet(M) $
with $(\Lambda^\bullet\mathfrak g^*, \delta)$ the Eilenberg-MacLane complex and $(\Omega^\bullet(M),d) $
the de Rham-complex. We then consider a natural cocycle $g$ of the ensuing total complex, already
constructed in \cite{MR1618735} for different reasons, and show that for a given action $\zeta$,
homotopy co-moment maps are in bijection with potentials of $g$ (Lemma 3.3 and Theorem 3.5). The last section 
specializes the preceding result to natural cases as symplectic manifolds and exact multisymplectic manifolds,
 and recovers existence results for partial homotopy co-moments: universal momentum maps in the sense of 
 \cite{forger-paufler-roemer},  covariant (multi)momentum maps in the sense of   \cite{MR1244450} and
 \cite{mm-clft}, and multi-moment maps in the sense of  \cite{MR3079342}.\\
 
\noindent  {\bf N.B.} The main results of this note make part of the M.Sc.~thesis of the first author \cite{ryvkin-msc}.
 The obstruction-theory for existence and unicity of co-moments was independently developed in 
 \cite{fregier14cohomology}. The latter reference is a sequel to \cite{fregier2013homotopy}, where a
 cohomological framework is already announced, and both of these articles are oriented towards 
 geometric structures related to higher categories and equivariant cohomology, whereas our work is more 
 focussed on applications to classical 
 multisymplectic geometry.
 
\section{The observable algebras of n-plectic manifolds}

\begin{Definition}
Let $n$ be in $\mathbb N$. A \textit{pre-n-plectic} manifold $(M,\omega)$ is a manifold $M$ equiped with a closed $(n{+}1)$-form $\omega\in \Omega^{n+1}_{cl}(M)$. A pre-n-plectic manifold $(M,\omega)$ is called \textit{n-plectic} or \textit{multisymplectic}  if the map  $\omega_p^\flat:T_pM\to \Lambda^nT^*_pM$ defined by $\omega^\flat_p(v)=\iota_v\omega_p$ is injective for all $p\in M$. (The symbol $\iota_v$ denotes
the contraction with $v$.)
\end{Definition}

\begin{Definition}
Let $(M,\omega)$ be a pre-n-plectic manifold. The \textit{space of observables of $(M,\omega)$ }is defined as the graded vector space $L(M,\omega)=L=\bigoplus_{i=-n+1}^0L_i$, where $L_i:=\Omega^{n-1+i}(M)$ for $i\neq 0$ and 
\begin{align*}
L_0:=\left\{(v,\alpha)\in \mathfrak X(M)\times \Omega^{n-1}(M)~|~ d\alpha=-\omega^\flat(v)\right\}.
\end{align*}
The image of the projection $\pi_{\mathfrak X}:L_0\to \mathfrak X(M)$ is called \textit{the space of Hamiltonian vector fields} and denoted by $\mathfrak X_{Ham}(M,\omega)$.  
\end{Definition}

\begin{Remark}
The vector space $L_0$ is the pullback fitting into the following diagram:
\begin{align}
\begin{xy}
\xymatrix{
L_0 \ar[d]^{\pi_{\Omega}}\ar[r]^{\pi_{\mathfrak X}}&\mathfrak X(M)\ar[d]^{-\omega^\flat}\\
 \Omega^{n-1}(M) \ar[r]^d& \Omega^{n}(M) }
\end{xy} \label{prepullback}
\end{align} 
When $(M,\omega)$ is n-plectic the map $\omega^\flat$ is injective so that $L_0$ is canonically isomorphic to the subspace $\Omega_{Ham}^{n-1}(M,\omega):=\pi_{\Omega}(L_0)$ of $\Omega^{n-1}(M)$ called \textit{the space of Hamiltonian forms of $(M,\omega)$}. In this case one usually writes  $\Omega_{Ham}^{n-1}(M,\omega)$  instead of $L_0$ for the degree zero component of the space of observables.
\end{Remark}

\begin{Remark}
For $(v,\alpha)\in L_0$ we have $\mathcal L_v\omega=d\iota_v\omega+\iota_vd\omega=d(-d\alpha)+0=0$ so that $\mathfrak X_{Ham}(M,\omega)$ is a subset of $\mathfrak X(M,\omega):=\{v\in \mathfrak X(M)~|~\mathcal L_v\omega=0\}$.
\end{Remark}

\noindent Unlike in the symplectic ({\it i.e.} $n=1$) case the observables of a pre-n-plectic manifold do not form a Lie algebra for $n>1$. 
They carry the more general structure of an $L_\infty$-algebra.

\begin{Definition}$ $
\begin{enumerate}
\item[(a)] An \textit{$L_\infty$-algebra} (or \textit{Lie-$\infty$-algebra}) is a graded vector space $L= \bigoplus_{i\in \mathbb Z}L_i$ together with a family of graded skew-symmetric multilinear maps $\{l_k:\varprod^kL\to L~|~k\in \mathbb N\}$ such that $l_k$ has degree $2{-}k$ and the following identity holds
\begin{align}\label{linfy}
	\sum_{i+j=n+1} (-1)^{i(j+1)}\sum_{\sigma\in ush(i,n-i)}sgn(\sigma)\epsilon(\sigma;x_1,...,x_n)~l_j(l_i(x_{\sigma(1)},..., x_{\sigma(i)}), x_{\sigma(i+1)}...,x_{\sigma(n)})=0
\end{align}
for all $n\in\mathbb N$, where $\epsilon(\sigma;x_1,...,x_n)$ denotes the Koszul sign of $\sigma$ acting on the elements $x_1,...,x_n$ and $ush(i,n{-}i)\subset S_n$ denotes the space of all $(i,n{-}i)$-unshuffles.
\item[(b)]  Let $m\in \mathbb N$. An $L_\infty$-algebra $L$ is called \textit{Lie m-algebra} if its underlying graded vector space is concentrated in the degrees $\{i\in \mathbb Z~|~-m+1\leq i \leq 0\}$.
\item[(c)]  A Lie m-algebra $L$ is called \textit{grounded}, if  $ l_k(x_1,...,x_k)$ is zero whenever $k>1$ and 
$\sum_{i=1}^k|x_i|\neq 0$. (Here and in the sequel $|x|$ denotes the degree of a homogeneous element $x$ of a graded vector space.)
\end{enumerate}
\end{Definition}

\begin{Remark}
As most terms in the multi-bracket equation (\ref{linfy}) vanish, a grounded Lie m-algebra can be described as an $m$-term cochain complex $\left(\bigoplus_{i=-m+1}^0L_i~,~ l_1\right)$ with a family of linear maps $\{l_k:\Lambda^k L_0\to L_{2-k}~|~1<k\leq m+1\}$ satisfying $l_k(l_1(x_1),x_2,...,x_k)=0$ for all $x_1,..,x_k\in L_0$ and $\partial l_{k}=l_1 l_{k+1}$ for $2\leq k\leq m+1$, where $l_{m+2}$ is to be interpreted as the zero map and $\partial: \Lambda^p(L_0,L_k)\to $ $\Lambda^{p+1}(L_0,L_k)$ is defined as follows, generalizing the Chevalley-Eilenberg differential in Lie algebra cohomology:
\begin{align*}
(\partial f)(x_1,...,x_{p+1})=\sum_{1\leq i<j\leq p+1} (-1)^{i+j}f(l_2(x_i, x_j), x_1,...,\hat x_i,...,\hat x_j,...,x_{p+1}).
\end{align*}
The notation $\Lambda^r(L_0,L_i)=Hom(\Lambda^rL_0,L_i)$ stands of course for the vector space of r-multilinear skew-symmetric maps from $\varprod^rL_0$ to $L_i$. 
\end{Remark}

\begin{Proposition}
Let $(M,\omega)$ be pre-n-plectic. Then $L=L(M,\omega)$ carries the structure of a grounded Lie n-algebra given by the maps $\{l_k~|~1\leq k\leq n+1\}$ defined as follows:
\begin{itemize}
	\item $l_1(\alpha):=d\alpha$ for $\alpha \in L_{-i}, i>1$, $l_1(\alpha)=(0,d\alpha)$ for $\alpha\in L_{-1}$ and $l_1((v,\alpha))=0$ for $(v,\alpha)\in L_0$
	\item $l_2((v,\alpha),(w,\beta))=([v,w],\iota_w\iota_v\omega)$
\item for $k>2$ we set $l_k((v_1,\alpha_1),...,(v_k,\alpha_k)):=-(-1)^{k(k+1)/2}\iota_{v_k}...\iota_{v_1}\omega$.
\end{itemize}
We call $(L(M,\omega),\{ l_k\})$ the \textit{Lie n-algebra of observables} of $(M,\omega)$.
\end{Proposition}
\begin{proof} A proof can be found in \cite{fregier2013homotopy} (Theorem 4.7), where the property of ``being grounded''  is referred to as ``having Property P''. In the n-plectic case the result follows immediately from \cite{MR2566161} for $n=2$ and from \cite{MR2892558} for $n>2$.
\end{proof}

\section{Lie algebras of n-plectic symmetries}
In this section we consider infinitesimal actions $\zeta:\mathfrak g\to \mathfrak X(M)$ of finite dimensional Lie algebras $\mathfrak g$ on pre-n-plectic manifolds $(M,\omega)$, preserving the form $\omega$, and define what is a homotopy co-moment
for such an action.

\begin{Definition}
An infinitesimal action $\zeta: \mathfrak g\to \mathfrak X(M)$ is called \textit{pre-n-plectic} or simply \textit{n-plectic} if the image of $\zeta$ is in $\mathfrak X(M,\omega)$. The action is called \textit{(weakly) Hamiltonian} if there exists a linear lift $j:\mathfrak g\to L=L(M,\omega)$ of $\zeta$ along $\pi_{\mathfrak X}$. Such a lift is called a \textit{(weak) co-moment map} for the action $\zeta$. The action is called \textit{strongly Hamiltonian} if there exists an $L_\infty$-morphism $F:\mathfrak g\to L$ lifting $\zeta$ along $\pi_{\mathfrak X}$. In the latter case the morphism $F$ is called a \textit{homotopy co-moment map} for $\zeta$. 
\end{Definition}

\begin{Remark}
A strongly Hamiltonionian action is always weakly Hamiltonian. A weakly Hamiltonian action is always n-plectic.
\end{Remark}

\begin{Remark}
If $(M,\omega)$ is a pre-symplectic manifold then the above definition reduces to the classical notion of (pre-)symplectic/Hamiltonian/strongly Hamiltonian actions and co-moment maps. 
\end{Remark}

\noindent For the convenience of the reader we explicitly recall the notion of $L_\infty$-morphisms:

\begin{Definition}\label{weakmorph}
An \textit{$L_\infty$-morphism} from $(L,l_i)$ to $(L',l_i')$ is a family $\{f_k\}$ of graded skew-symmetric maps of degrees $1{-}k$ satisfying the following condition for $n\geq 1$:

\begin{align*}
&\sum_{i+j=n+1} \sum_{\sigma\in ush(i,n-i)}(-1)^{i(j+1)}sgn(\sigma)\epsilon(\sigma;x_1,...,x_n)
 f_{j}\left(
\left( l_i(x_{\sigma(1)},..., x_{\sigma(i)})\right), x_{\sigma(i+1)},..., x_{\sigma(n)}
\right)\\
&=\sum_{p=1}^n
\sum_{\substack{\sum_{j=1}^p k_j=n\\ k_i\leq k_{i+1}}}
\sum_{\substack{\sigma\in ush(k_1,...,k_p)\\ \sigma(\sum_{i=1}^{j-1}k_i+1)<\sigma(\sum_{i=1}^{j}k_i+1)\\ \text{whenever } k_j=k_{j+1}}}
(-1)^\beta sgn(\sigma)\epsilon(\sigma;x_1,...,x_n)~~~~\times \\
&~~~~~~~~~~~~~~~~~~~ l_p'(
 f_{k_1}(x_{\sigma(1)},..., x_{\sigma(k_1)}) 
  ,  f_{k_2}(x_{\sigma(k_1+1)},..., x_{\sigma(k_1+k_2)}), ...,  f_{k_p}(x_{\sigma(n-k_p+1)},..., x_{\sigma(n)}) ),
\end{align*}
where $\beta$ is given by the following formula:
\begin{align*}
\beta=&\frac{p(p-1)}{2}+\sum_{i=1}^{p}k_i(p-i)+\\ &(k_p-1)\sum_{i=1}^{n-k_p}|x_{\sigma(i)}|+ (k_{p-1}-1)\sum_{i=1}^{n-(k_p+k_{p-1})}|x_{\sigma(i)}|+...+ (k_2-1)\sum_{i=1}^{n-(k_p+k_{p-1}+...+ k_2)}|x_{\sigma(i)}| .
\end{align*}
\end{Definition}

\begin{Remark}
An $L_\infty$-algebra can equivalently be described as the free symmetric coalgebra of the suspension of the graded vector space $L$ equipped with a coderivation that squares to zero. The above multi-bracket condition for morphisms then corresponds to the preservation of the coderivations. The equivalence of the two approaches is proven in \cite{MR1235010}. 
\end{Remark}

\begin{Proposition}
Let $\mathfrak g$ be a Lie algebra and $L$ a grounded Lie n-algebra. Then an $L_\infty$-morphism from $\mathfrak g$ to $L$ is a family of maps $\{f_1,...,f_n\}$ such that $f_k$ has degree $1{-}k$ and is satisfying $\delta f_k + l_1 f_{k+1}=$ $-f^*_1l_k$ for $1<k\leq n$, where $f_{n+1}$ is interpreted as the zero map and $\delta:\Lambda^k(\mathfrak g,L)\to$ $\Lambda^{k+1}(\mathfrak g,L)$ denotes the \textit{Chevalley-Eilenberg differential (with respect to the trivial representation)} defined as follows:
\begin{align*}
(\delta f)(X_1,...,X_{k+1})=\sum_{1\leq i<j\leq n} (-1)^{i+j}f([X_i, X_j], X_1,...,\hat X_i,...,\hat X_j,...,X_{k+1}).
\end{align*}
\end{Proposition}

\begin{proof}
The result is easily deduced from Definition \ref{weakmorph} and its proof is explicitly given in \cite{fregier2013homotopy}, Proposition 3.8.
\end{proof}

\section{The existence and unicity of co-moment maps}
To analyse the cohomological obstruction to the existence of homotopy co-moment maps, we study the double complex $\Lambda^\bullet\mathfrak g^*\otimes \Omega^\bullet(M) $ resulting from tensoring the cochain complexes $(\Lambda^\bullet\mathfrak g^*,~\delta )$, where $\delta$ denotes the Chevalley-Eilenberg differential, with $(\Omega^\bullet(M),d)$. Diagrammatically we consider thus:

\begin{xy}
\xymatrix{...&...&...&...&...\\
\Lambda^3\mathfrak g^*\otimes \Omega^{0}(M)\ar[r]^{id\otimes d}\ar[u]^{\delta\otimes id}&\Lambda^3\mathfrak g^*\otimes \Omega^{1}(M)\ar[r]^{id\otimes d}\ar[u]^{\delta\otimes id}&\Lambda^3\mathfrak g^*\otimes \Omega^{2}(M)\ar[r]^{id\otimes d}\ar[u]^{\delta\otimes id}&\Lambda^3\mathfrak g^*\otimes \Omega^{3}(M)\ar[r]\ar[u]^{\delta\otimes id}& ... \\
\Lambda^2\mathfrak g^*\otimes \Omega^{0}(M)\ar[r]^{id\otimes d}\ar[u]^{\delta\otimes id}&\Lambda^2\mathfrak g^*\otimes \Omega^{1}(M)\ar[r]^{id\otimes d}\ar[u]^{\delta\otimes id}&\Lambda^2\mathfrak g^*\otimes \Omega^{2}(M)\ar[r]^{id\otimes d}\ar[u]^{\delta\otimes id}&\Lambda^2\mathfrak g^*\otimes \Omega^{3}(M)\ar[r]\ar[u]^{\delta\otimes id}& ... \\
\Lambda^1\mathfrak g^*\otimes \Omega^{0}(M)\ar[r]^{id\otimes d}\ar[u]^{\delta\otimes id}&\Lambda^1\mathfrak g^*\otimes \Omega^{1}(M)\ar[r]^{id\otimes d}\ar[u]^{\delta\otimes id}&\Lambda^1\mathfrak g^*\otimes \Omega^{2}(M)\ar[r]^{id\otimes d}\ar[u]^{\delta\otimes id}&\Lambda^1\mathfrak g^*\otimes \Omega^{3}(M)\ar[r]\ar[u]^{\delta\otimes id}& ...\\
\Lambda^0\mathfrak g^*\otimes \Omega^{0}(M)\ar[r]^{id\otimes d}\ar[u]^{\delta\otimes id}&\Lambda^0\mathfrak g^*\otimes \Omega^{1}(M)\ar[r]^{id\otimes d}\ar[u]^{\delta\otimes id}&\Lambda^0\mathfrak g^*\otimes \Omega^{2}(M)\ar[r]^{id\otimes d}\ar[u]^{\delta\otimes id}&\Lambda^0\mathfrak g^*\otimes \Omega^{3}(M)\ar[r]\ar[u]^{\delta\otimes id}& ... }
\end{xy}$ $\newline

\noindent We turn $\Lambda^\bullet\mathfrak g^*\otimes \Omega^\bullet(M)$ into a singly graded cochain complex $(C^\bullet, D)$ by defining the total degree of an element as the sum of the individual degrees, {\it i.e.} we set $C^k=\bigoplus_{i+j=k} \Lambda^i\mathfrak g^*\otimes \Omega^j(M)$. To assure that $D^2=0$ we have to alter the sign of one of the differentials. We stick to the following convention:
\begin{align*}
D|_{ \Lambda^i\mathfrak g^*\otimes \Omega^j(M) }=(\delta\otimes id) +(-1)^n(-1)^{i+j}(id\otimes d).
\end{align*}
As usually, we often abusively write $\delta(\alpha_i\otimes \eta_j)$ for   $\delta\alpha_i\otimes \eta_j$ and  $d(\alpha_i\otimes \eta_j)$ for  $\alpha_i\otimes d\eta_j$ in the sequel.\newline

\noindent We will now analyse an action $\zeta:\mathfrak g\to \mathfrak X(M,\omega)$ in terms of this complex. We define the maps $g_k\in\Lambda^k(\mathfrak g,\Omega^{n+1-k}(M)) = \Lambda^k\mathfrak g^*\otimes \Omega^{n+1-k}(M)$ for $1\leq k\leq n+1$ by 
\begin{align*}
g_k(X_1,...,X_k):=-(-1)^{k(k+1)/2}\iota_{\zeta(X_k)}...\iota_{\zeta(X_1)}\omega.
\end{align*}
\begin{Lemma}
The sum of the above defined classes $g=\sum_{i=1}^{n+1}g_i$ is a $D$-cocycle {\it i.e.} $D(g)=0$.
\end{Lemma}
\begin{proof}
Obviously $D(g)=0$ is equivalent to $\delta g_k=dg_{k+1}$ for $0\leq k\leq n+1$, where $g_0$ and $g_{n+2}$ are interpreted as the zero maps. The latter equalities follow easily from $\mathcal L_{\zeta(X)}\omega=0$ for all $X\in\mathfrak g$ (cf. Lemma 9.3. of \cite{fregier2013homotopy}).
\end{proof}

\begin{Remark}
Up to signs the cocycle $g$ corresponds to the Cartan cocycle $\hat \omega$ constructed for different purposes in Chapter 2, §4 of  \cite{MR1618735}.
\end{Remark}

\begin{Lemma}\label{lem1}
There is a one-to-one correspondance between homotopy co-moment maps and $D$-potentials $p=\sum_{i=1}^{n} p_i\in C^n$ with $p_i \in \Lambda^i\mathfrak g^*\otimes \Omega^{n-i}(M)$ of $g$.
\end{Lemma}

\begin{Remark}
We do not have a $p_0$ term as $\delta|_{\Lambda^0\mathfrak g^*}=0$ and $dp_0$ would be in $\Lambda^0\mathfrak g^*\otimes \Omega^{n+1}(M)$. Since $\mathfrak g$ has no component in this bidegree, there is no need for a $p_0$.
\end{Remark}

\begin{proof}\textit{\textcolor{black}{of Lemma \ref{lem1}}}. Let $F=\{f_k~|~1\leq k\leq n\}$ be a homotopy co-moment map. Then setting $p_1:= -\pi_{\Omega}\circ f_1$ and $p_k:=-f_k$ for $k\geq 2$ we claim that $p=\sum p_k $ is a potential for $g$. As only two components of $D(p)$ have non-vanishing projections on $\Lambda^k(\mathfrak g,\Omega^{n+1-k}(M))$ it suffices to show that $\delta p_k+dp_{k+1}=g_{k+1}$. Since the $L_\infty$-morphism $F$ lifts $\zeta$ along $\pi_{\mathfrak X}$ we have $\pi_{\mathfrak X}f_1(X)=\zeta(X)$, which in turn implies $df_1(X)=-\iota_{\zeta(X)}\omega$ i.e $dp_1=g_1$. Further we observe that $g_{k+1}(X_1,...,X_{k+1})=f^*_1l_{k+1}(X_1,...,X_{k+1})$, which directly implies $\delta p_k+dp_{k+1}=g_{k+1}$ for $k>0$.\newline\newline
Let now $p=p_1+...+p_n$ with $p_k\in\Lambda^k\mathfrak g^*\otimes \Omega^{n-k}(M)$ be a potential of $g$. We define $f_k:=-p_k$ for $k>1$ and $f_1(X)=(-p_1(X),\zeta(X))$. As $-d\pi_{\Omega} f_1(X)=dp_1(X)=g_1(X)=\iota_{\zeta(X)}\omega$  the map $f_1$ is well-defined and we can express $g_k$ by $f_1^*l_k$ for $k>1$. The higher identities directly follow from $\delta p_k+dp_{k+1}=g_{k+1}=f_1^*l_{k+1}$.
\end{proof}

\noindent The following theorem clarifies the existence and unicity question for homotopy co-moment maps.
\begin{Theorem}\label{mainthm}
Let $\zeta:\mathfrak g\to \mathfrak X(M,\omega)$ be an infinitesimal n-plectic action. The action is strongly Hamiltonian if and only if the above-defined cocycle $g$ vanishes in the cohomology of $C^\bullet$, {\it i.e.} $[g]=0$ as an element of $H^{n+1}(C^\bullet)$. If a homotopy co-moment map exists it is unique up to $D$-cocycles of total degree $n$ with vanishing $\Lambda^0\mathfrak g^*\otimes \Omega^{n}(M)$-component.
\end{Theorem}
\begin{proof}
The statement of the theorem follows directly from Lemma \ref{lem1}.
\end{proof}

\noindent The element $g_1$ satisfies $dg_1=0$. Thus it is a $d$-cocycle and we can regard the cohomology class $[g_1]_{dR}\in \Lambda^1\mathfrak g^*\otimes H^n_{dR}(M)=Hom(\mathfrak g, H^n_{dR}(M))$. If this element is zero there exists a potential $(-j)$ in $Hom(\mathfrak g, \Omega^{n-1}(M))$. As $d(j(X))=-g_1(X)=-\iota_{\zeta(X)}\omega$ we can conclude that $j$ is a linear lift of $\zeta$ along $\pi_{\mathfrak X}$. This leads us to 
 
\begin{Corollary}\label{g1cor}
Let $\zeta:\mathfrak g\to \mathfrak X(M,\omega)$ be an infinitesimal pre-n-plectic action. The action is weakly Hamiltonian if and only if $[g_1]_{dR}=0$ in $ \mathfrak g^*\otimes H^{n}_{dR}(M)$.
\end{Corollary}

\noindent This corollary can be also interpreted in terms of the following exact sequence of $L_\infty$-algebras:

\begin{align}\label{d1}
\begin{xy}
\xymatrix{
0\ar[r]&\bigoplus_{i=0}^{n-2}\Omega^i(M)\oplus \Omega_{cl}^{n-1}(M)\ar[r]&L(M,\omega)\ar[r]^{\pi_\mathfrak X\circ \pi_{L_0}}&\mathfrak X(M,\omega)\ar[r]^{\gamma}_{v\mapsto [\iota_v\omega]}&H_{dR}^{n}(M)&\\
&&& \mathfrak g\ar[u]^\zeta\ar@{-->}[ul]^j&& }
\end{xy}
\end{align}
where $\Omega^{n-1}_{cl}(M)$ denotes the space of closed $(n{-}1)$-forms on $M$, the only nontrivial bracket of the leftmost space is the unary bracket $d$ and the rightmost term is to be interpreted as an abelian Lie algebra. The map $\zeta$ can be lifted to a linear map $j$ if and only if $([g_1]_{dR}=)~\gamma\zeta =0$. \newline\newline
\noindent By the Kuenneth theorem (cf., eg., \cite{MR1269324}, Thm 3.6.3) the total cohomology group $H^{n+1}(C^\bullet) $ is isomorphic to the direct sum $\bigoplus_{i+j=n+1}H^i(\mathfrak g)\otimes H^j_{dR}(M)$. Thus $[g]$ can be decomposed into classes $h_i\in H^{i}(\mathfrak g)\otimes H^{n+1-i}_{dR}(M)$ beyond $h_0=0$. These classes correspond to the $g_i$, but in a slightly subtle way, as the latter are in general neither $\delta$-closed nor $d$-closed. In order to interpret the $g_i$ as cocycles representing the $h_i$ both equivalence relations (de Rham and Chevalley-Eilenberg) have to be divided out at the same time, as $dg_i\in Im(\delta)$ and $\delta g_i\in Im(d)$.\newline\newline %
 In order to formalize this, we introduce the quotient maps $q_\Lambda:(\Lambda^k{\mathfrak g^*})\to \frac{\Lambda^k{\mathfrak g^*}}{(\Lambda^k{\mathfrak g^*})_{ex}}$ and $q_{\Omega}:\Omega^{l}(M)\to$ $\frac{\Omega^{l}(M)}{\Omega_{ex}^{l}(M)}$, where for a cochain complex $N$ we denote by $N^k_{ex}\subset N^k$ the subspace of exact elements. As $\delta$ resp. $d$ are zero on $(\Lambda^k{\mathfrak g^*)_{ex}}$ resp. $\Omega_{ex}^{l}(M)$ they induce maps $\bar \delta$ resp. $\bar d$ rendering the following sequences exact:  
\begin{align*}
\xymatrix{0\ar[r]& H^k(\mathfrak g)= \frac{(\Lambda^k{\mathfrak g^*})_{cl}}{(\Lambda^k{\mathfrak g^*})_{ex}}  \ar[r]^{~~~\subset}&  \frac{(\Lambda^k{\mathfrak g^*})}{(\Lambda^k{\mathfrak g^*})_{ex}} \ar[r]^{\bar \delta}&(\Lambda^{k+1}{\mathfrak g^*})_{ex}\ar[r]&0}
\end{align*}
resp.
\begin{align*}
\xymatrix{0\ar[r]& H_{dR}^l(M)= \frac{\Omega_{cl}^{l}(M)}{\Omega_{ex}^{l}(M)}  \ar[r]^{~~~\subset}&  \frac{\Omega^{l}(M)}{\Omega_{ex}^{l}(M)} \ar[r]^{\bar d}&\Omega_{ex}^{l+1}(M)\ar[r]&0~, }
\end{align*}
where $\subset$ denotes here the canonical inclusion maps. By the exactness of the preceeding two diagrams an element $a\in  \frac{(\Lambda^k{\mathfrak g^*})}{(\Lambda^k{\mathfrak g^*})_{ex}} \otimes \frac{\Omega^{l}(M)}{\Omega_{ex}^{l}(M)} $ satisfying $(\bar \delta\otimes id)a=0$ and $(id\otimes \bar d)a=0$ already fulfills $a\in H^k(\mathfrak g)\otimes H^l_{dR}(M)$. For $k\geq 2$,   $(q_\Lambda\otimes q_\Omega)(g_k)$ satisfies these equations and consequently we can regard $(q_\Lambda\otimes q_\Omega)(g_k)=h_k$ as an element of $H^k(\mathfrak g)\otimes H^{n+1-k}_{dR}(M)$.

\section{Applications and special cases}
\subsection{Co-moment maps in symplectic geometry}
When $\omega$ is a 2-form we recover the following classical results:
\newline
\newline
Let $(M,\omega)$ be a symplectic manifold and $\zeta:\mathfrak g\to \mathfrak X(M,\omega)$ an infinitesimal symplectic group action. The action is weakly Hamiltonian ({\it i.e.} there exists a linear lift $j:\mathfrak g\to C^\infty(M)$ of $\zeta$ along  $v_\bullet=-(\omega^\flat)^{-1}(d\bullet): C^\infty(M)\to \mathfrak X(M,\omega)$) if and only if $h_1=(q_\Lambda\otimes q_\Omega)(g_1)$ vanishes, {\it i.e.} if the image of $\zeta$ is a subset of the kernel of $\gamma = [\iota_{\bullet}\omega]_{dR}:\mathfrak X(M,\omega)\to H^1_{dR}(M)$. Moreover, as $\omega^\flat$ is invertible, $\gamma$ is surjective and in this case Diagram (\ref{d1}) can be extended to the following exact sequence of Lie algebras:

\begin{align*}
\begin{xy}
\xymatrix{
0\ar[r]& C^\infty_{cl}(M)\ar[r]&C^\infty(M)\ar[r]^{v_\bullet }&\mathfrak X(M,\omega)\ar[r]^{\gamma}_{x\mapsto [\iota_x\omega]}&H_{dR}^{1}(M)\ar[r]&0\\
&&& \mathfrak g\ar[u]^\zeta\ar@{-->}[ul]^j& &}
\end{xy}
\end{align*}
Furthermore, if $h_1=0$ and $j$ is a map as above, the class $h_2$ can be represented by the cocycle $\bar j=((X,Y)\mapsto \iota_{\zeta(Y)}\iota_{\zeta(X)}\omega-j([X,Y]))\in (\Lambda^2\mathfrak g^*)_{cl}\otimes C^{\infty}_{cl}(M)$. Hence $[\bar j]_{CE}=h_2\in H^2(\mathfrak g)\otimes H^0_{dR}(M)$ is the obstruction for the weakly Hamiltonian action defined by $\zeta$ and $j$ to be strongly Hamiltonian.
\subsection{Exact multisymplectic manifolds}
\begin{Lemma}\label{exactlemma}
Let $\zeta:\mathfrak g\to \mathfrak X(M,\omega)$ be an infinitesimal pre-n-plectic action. If $\omega$ has a $\mathfrak g$-invariant potential $\eta$ then the action is strongly Hamiltonian and the $L_\infty$-morphism takes the form $f_1(X)=(\zeta(X),\iota_{\zeta(X)}\eta)$ and for $k\geq 2$, $f_k(X_1,...,X_k)=(-1)^{k}(-1)^{\frac{k(k+1)}{2}}\iota_{\zeta(X_k)}...\iota_{\zeta(X_1)} \eta$.
\end{Lemma}
\begin{proof}
A proof is provided in  \cite{fregier2013homotopy} (Lemma 8.1 and Theorem 6.3) or can easily be verified upon using the fact, that the $\{f_k\}$ are constructed from $\eta$ in the same way, as the $\{g_k\}$ are constructed from $\omega$.
\end{proof}
\noindent In Section 4 of \cite{forger-paufler-roemer}, $\mathfrak X_{EH}(M)$ is defined as the space of vector fields $v$ on an exact n-plectic manifold $(M,\omega)$ such that $\mathcal{L}_v\eta=0$. The map $J:\mathfrak X_{EH}(M)\to 
\Omega^{n-1}_{Ham}(M,\omega)$ defined by $J(v)=\iota_v\eta$ is then called the ``universal momentum map''. Obviously, 
we have  $\pi_\Omega \circ f_1=J\circ\zeta:\mathfrak g\to \Omega^{n-1}_{Ham}(M,\omega)$ in the notations of Lemma \ref{exactlemma}, and in the $n$-plectic case $f_1$ is, of course, already determined by $\pi_\Omega \circ f_1$. 
\subsection{Simple Lie groups} Let $G$ be a connected simple Lie group with Lie algebra $(\mathfrak g, [\cdot,\cdot])$ and $\langle\cdot,\cdot\rangle$ its (invariant, non-degenerate) Killing-form. Setting $\omega_e(X,Y,Z)=\langle X,[Y, Z] \rangle$ we get a three-form $\omega_e$ on $\mathfrak g$. Using the non-degeneracy of the Killing-form and $[\mathfrak g, \mathfrak g]=\mathfrak g$ one shows that $\omega_e$ is non-degenerate. It extends to to a bi-invariant form $\omega\in \Omega^3(G)$,
that is closed and non-degenerate by construction and thus defines a 2-plectic structure on $G$.\newline

\noindent The group $G$ acts on itself from the left via $(g,x)\mapsto g\cdot x$. The corresponding infinitesimal action $\zeta$ extends an $X\in \mathfrak g$ to a right-invariant vector field $\zeta(X)$ on $G$. We recall that $H^1(\mathfrak g)=H^2(\mathfrak g)=0$ and $[\omega_e]_{CE}=[\langle\cdot,[\cdot,\cdot]\rangle]_{CE}\neq 0$ in $H^3(\mathfrak g)$, since $\mathfrak g$ is simple. By the Kuenneth theorem  $[g]\in H^{3}(C^\bullet)$ can be decomposed into $h_1\in H^1(\mathfrak g)\otimes H^2(G)$, $h_2 \in H^2(\mathfrak g)\otimes H^1(G)$ and $h_3\in  H^3(\mathfrak g)\otimes H^0(G)= H^3(\mathfrak g)$. As $H^1(\mathfrak g)$ and $H^2(\mathfrak g)$ vanish, we have $h_1=h_2=0$. As $[g_1]_{dR}=h_1=0$ Corollary \ref{g1cor} implies that $\zeta$ is weakly Hamiltonian. Let us now consider $h_3=[g_3]_{CE}\in H^3(\mathfrak g)$. We have $g_3(X,Y,Z)=\omega(\zeta(X),\zeta(Y),\zeta(Z))=$\-
$\omega_e(X,Y,Z)$. Thus $[g_3]_{CE}=[\omega_e]_{CE}\neq 0\in H^3(\mathfrak g)$ and $\zeta$ is not strongly Hamiltonian.

\subsection{Vanishing cohomology groups}
In \cite{fregier2013homotopy} a class $[c_p]\in H^{n+1}(\mathfrak g)$ is constructed for any point $p\in M$ of a pre-n-plectic manifold $(M,\omega)$ with a weakly Hamiltonian action $\zeta$ and a weak co-moment map $j$ by the formula $c_p(X_1,...,X_{n+1}):=-(-1)^n(-1)^{n(n+1)/2}\omega(\zeta(X_1),...,\zeta(X_{n+1}))|_p$. This class is equal to 
\begin{align*}
(-1)^ni^*h_{n+1}\in H^{n+1}(\mathfrak g)\otimes H^0_{dR}(\{p\})= H^{n+1}(\mathfrak g),
\end{align*}
 where $i^*: H^{n+1}(\mathfrak g)\otimes H^0_{dR}(M)\to H^{n+1}(\mathfrak g)\otimes H^0_{dR}(\{p\})$ is induced by the inclusion $i:\{p\}\to M$. If $M$ is connected $i^*$ is an isomorphism and we obtain Theorem 9.7 of \cite{fregier2013homotopy} as a corollary of Theorem \ref{mainthm}:

\begin{Corollary}
Let $(M,\omega)$ pe a connected pre-n-plectic manifold and $(\mathfrak g, \zeta, j)$ a weakly Hamiltonian action of $\mathfrak g$ on $M$. If $H^k_{dR}(M)$ vanish for $k\in \{1,..., {n-1}\}$ and $[c_p]=0$ for any $p\in M$, then there exists a homotopy co-moment map for $\zeta:\mathfrak g\to \mathfrak X(M,\omega)$.
\end{Corollary}

\subsection{Covariant momentum maps}
Given an n-plectic manifold $(M,\omega)$ and an infinitesimal n-plectic action $\zeta:\mathfrak g\to \mathfrak X(M,\omega)$, a smooth section $J:M\to \Lambda^{n-1}T^*M\otimes \mathfrak g^*$ is called a ``covariant momentum map (or multimomentum map)'' in Section 4 C of \cite{mm-clft}, if $d(\langle J,X\rangle)=\iota_{\zeta(X)}\omega$ for all $X\in \mathfrak g$. We thus find that $J$ corresponds to  $f_1:\mathfrak g\to \Omega^{n-1}_{Ham}(M,\omega)$,  since
$(\pi_{\Omega}\circ f_1(X))(p)=\langle J,X\rangle(p)=\langle J(p),X\rangle$ $\forall X\in \mathfrak g, ~\forall p\in M$ and
$\pi_{\Omega}\circ f_1$ determines $f_1$ in the $n$-plectic case. If, furthermore, $\omega$ has an invariant potential $\eta$, it is noted in  \cite{mm-clft} that $J(X)=\iota_{\zeta(X)}\eta$ is a covariant momentum map called ``special''. This corresponds to the 
first component of the $L_\infty$-morphism explicitly given in Lemma \ref{exactlemma}.

\subsection{Covariant multimomentum maps}
\noindent Let $(M,\omega)$ be a pre-n-plectic manifold. We consider the following short exact sequence  of Lie n-algebras:
\begin{align*}
\xymatrix{
0\ar[r]& \bigoplus_{i=0}^{n-2}\Omega^i(M) \oplus \Omega^{n-1}_{cl}(M)\ar[r]& L(M,\omega) \ar[r]&\mathfrak X_{Ham}(M,\omega)\ar[r]& 0}
\end{align*}
and mod out the following exact subsequence

\begin{align*}
\xymatrix{
0\ar[r]& \bigoplus_{i=0}^{n-2}\Omega^i(M) \oplus \Omega^{n-1}_{ex}(M)\ar[r]& \bigoplus_{i=0}^{n-2}\Omega^i(M)\oplus\left(  \Omega^{n-1}_{ex}(M)\oplus ker (\iota_\bullet \omega)\right)\ar[r]&ker (\iota_\bullet \omega)\ar[r]& 0}
\end{align*}

\noindent to obtain the following exact sequence of Lie algebras:
\begin{align*}
\xymatrix{
0\ar[r]& H^{n-1}_{dR}(M)\ar[r]&\frac{L_0}{\Omega^{n-1}_{ex}(M)\oplus ker(\iota_{\bullet}\omega)}\ar[r]&\frac{\mathfrak X_{Ham}(M,\omega)}{ker(\iota_{\bullet}\omega)}\ar[r]& 0 \,\, . &}
\end{align*}
Here the notation $ker (\iota_\bullet \omega)$ is also used for $\{(0,v)~|~\iota_v\omega=0\}\subset L_0$, isomorphic to $ker (\iota_\bullet \omega)\subset \mathfrak X_{Ham}(M,\omega)$ via $\pi_\mathfrak X$. Note that this subspace has trivial intersection with $\{(0,\alpha)~|~d\alpha=0\}\subset L_0$, isomorphic to the space of closed forms $\Omega^{n-1}_{cl}(M)$.  Notably $ker(\iota_\bullet \omega)\cap \Omega^{n-1}_{ex}(M)=\{0\}$ inside $L_0$.\newline\newline
Let $\zeta:\mathfrak g\to \mathfrak X_{Ham}(M,\omega)=ker(\gamma)\subset \mathfrak X(M,\omega)$ be an infinitesimal n-plectic action. Since the sequence is exact, there always exists a linear lift $\bar j$ of $\bar \zeta=\pi\circ \zeta$ along $\frac{L_0}{\Omega^{n-1}_{ex}(M)\oplus ker(\iota_{\bullet}\omega)}\to\frac{\mathfrak X_{Ham}(M,\omega)}{ker(\iota_{\bullet}\omega)}$, where $\pi:\mathfrak X_{Ham}(M,\omega)\to \frac{\mathfrak X_{Ham}(M,\omega)}{ker(\iota_{\bullet}\omega)}$ is the projection. Putting this together, we get the following diagram, where the top row is exact and all morphisms (except possibly $\bar j$) are Lie algebra morphisms:
\begin{align*}
\xymatrix{
0\ar[r]& H^{n-1}_{dR}(M)\ar[r]&\frac{L_0}{\Omega^{n-1}_{ex}(M)\oplus ker(\iota_{\bullet}\omega)}\ar[r]&\frac{\mathfrak X_{Ham}(M,\omega)}{ker(\iota_{\bullet}\omega)}\ar[r]& 0\\
&&&\mathfrak g\ar[u]^{\bar \zeta}\ar[ul]^{\bar j}&}
\end{align*}

\noindent which corresponds to the diagram on page 19 of \textcolor{black}{ \cite{MR1244450}}. The dual of such a map $\bar j$ is then a \textit{covariant multimomentum map} in the sense of \textcolor{black}{Cari\~nena, Crampin and Ibort}. Writing $\bar l_2$ for the Lie bracket on $\frac{L_0}{\Omega^{n-1}_{ex}(M)\oplus ker(\iota_{\bullet}\omega)}$ as in \textcolor{black}{ \cite{MR1244450}} one defines the cocycle $\bar c=((X,Y)\mapsto\bar l_2(\bar j(X),\bar j(Y))-\bar j([X,Y]))\in $ $\Lambda^2(\mathfrak g, H^{n-1}_{dR}(M))= \Lambda^2(\mathfrak g^*)$ $\otimes H^{n-1}_{dR}(M)$.  The cohomology class  $[\bar c]_{CE}\in H^2(\mathfrak g)\otimes H^{n-1}_{dR}(M)$ of this cocycle is the obstruction against a Lie algebra homomorphism $f$ lifting $\bar \zeta$. We claim that $[\bar c]_{CE}=h_2$. To show this we put $ c:=j^* l_2+\delta j=((X,Y)\mapsto l_2( j(X), j(Y))- j([X,Y]))$, where $j:\mathfrak g \to \Omega^{n-1}_{Ham}(M,\omega)$ is a lift of $\bar j$. Using the fact that $j^*l_2=g_2$ we calculate
\begin{align*}
[\bar c]_{CE}=[(id\otimes q_{\Omega})(c)]_{CE}=(q_{\Lambda}\otimes id)(id\otimes q_{\Omega})(c)=(id\otimes q_{\Omega}) (q_{\Lambda}\otimes id)(j^*l_2+\delta j)&\\=(id\otimes q_{\Omega}) (q_{\Lambda}\otimes id)(j^*l_2)=(q_{\Lambda}\otimes q_{\Omega})(j^*l_2)=(q_{\Lambda}\otimes q_{\Omega})(g_2)&=h_2.
\end{align*}
Thus we proved the following result of \cite{MR1244450}.

\begin{Corollary}
Let $\zeta:\mathfrak g\to \mathfrak X(M,\omega)$ be a weakly Hamiltonian infinitesimal action. Then there exists a Lie algebra homomorphism $f:\mathfrak g \to \frac{L_0}{\Omega^{n-1}_{ex}(M)\oplus ker(\iota_{\bullet}\omega)}$ which is the dual of a covariant multimomentum map if and only if the second obstruction class, $h_2$, vanishes.
\end{Corollary}

\subsection{Multi-moment maps}
Upon restricting to the top component of a homotopy co-moment map, one obtains the following notion of \textcolor{black}{Madsen and Swann (\cite{MR3079342})}.\newline\newline
Let $M,\omega$ be an n-plectic manifold. A \textit{multi-moment map} is a map $v: M \to P_{\mathfrak g}^*$, where $P_{\mathfrak g}=P_{\mathfrak g, n}\subset \Lambda^n\mathfrak g$ is the kernel of $\delta^*:\Lambda^n\mathfrak g\to \Lambda^{n-1}\mathfrak g$, satisfying certain conditions. It can be equivalently understood as a map $\overline v:P_\mathfrak g \to C^\infty(M)$ fulfilling 
\begin{enumerate}[(i)]
	\item $d(\bar v(p))=\iota_{\zeta(p)}\omega$ for all $p\in P_\mathfrak g$
	\item $\bar v(ad_X(p))=\mathcal L_{\zeta(X)}(\bar v(p))$ for all $X\in \mathfrak g$ and $p\in P_\mathfrak g$.
\end{enumerate}
Here we define for $X,X_j^\alpha\in \mathfrak g$ and $p=\sum_{\alpha}X_1^\alpha\wedge ...\wedge X_n^\alpha\in \Lambda^n\mathfrak g$, 
\begin{align*}
\iota_{\zeta(p)}\omega:=\sum_\alpha \iota_{\zeta(X_n^\alpha)}... \iota_{\zeta(X_1^\alpha)}\omega~~~~~\text{ and }~~~~ad_X(p):=\left(\Lambda^n(ad_X)\right)(p)~~.
\end{align*}
Given a strong co-moment $\{f_1,...,f_n\}$, it is easy to check that $\bar v:=-(-1)^{n(n+1)/2}f_n|_{P_\mathfrak g}$ fulfills the two preceeding conditions. In fact, one only needs the conditions $\delta f_n=-g_{n+1}$ and $\delta f_{n-1}+df_n=-g_n$ here. Hence, for a multi-moment map in the above sense to exist we need not require $h_i=0$ for all $i$, but only $h_{n+1}=0$ and $h_{n}=0$ suffice. Let us assume $h_{n+1}=0$ and $h_{n}=0$ and construct $f_n$ and $f_{n-1}$ ``by hand''. If 
\begin{align*}
0=h_{n+1}=(q_\Lambda\otimes q_{\Omega})(g_{n+1})\in H^{n+1}(\mathfrak g)\otimes H^0_{dR}(M)= H^{n+1}(\mathfrak g)\otimes C^{\infty}_{cl}(M),
\end{align*}
there exists a $\delta$-potential $\hat f_n\in \Lambda^{n+1}\mathfrak g^*\otimes C^\infty_{cl}(M)$ of $-g_{n+1}$. It holds that 
 \begin{align*}
(\bar \delta \otimes id)(q_\Lambda\otimes id)(-g_n-d\hat f_n)&=(-\delta g_n -\delta d\hat f_n)\\
=(-dg_{n+1} -d\delta \hat f_n)&=(-dg_{n+1} +dg_{n+1})=0\in \Lambda^{n+1}\mathfrak g^*\otimes \Omega^1(M).
\end{align*}
Thus $(q_\Lambda\otimes id)(-g_n-d\hat f_n)$ is an element of $H^n(\mathfrak g)\otimes \Omega^1(M)\subset\frac{\Lambda^n\mathfrak g^*}{(\Lambda^n\mathfrak g^*)_{ex}}\otimes \Omega^1(M) $. Next we observe that 
 \begin{align*}
(id \otimes \bar d)(q_\Lambda\otimes id)(-g_n-d\hat f_n)=(q_\Lambda\otimes id)(-dg_n)=(q_\Lambda\otimes id)(-\delta g_{n-1})=0\in H^n(\mathfrak g) \otimes \Omega^{2}(M).
\end{align*}
Consequently $(q_\Lambda\otimes id)(-g_n-d\hat f_n)$ is an element of $H^n(\mathfrak g)\otimes \Omega^{1}_{cl}(M)$. By the exactness of the sequence 
\begin{align*}
\xymatrix{H^{n}(\mathfrak g)\otimes C^\infty(M)\ar[r]^{id\otimes d}&H^{n}(\mathfrak g)\otimes \Omega^{1}_{cl}(M) \ar[r]^{id\otimes q_\Omega} & H^{n}(\mathfrak g)\otimes H^1_{dR}(M) }
\end{align*}
and $(id\otimes q_\Omega)(q_\Lambda\otimes id)(-g_n-d\hat f_n)=(q_{\Lambda}\otimes q_{\Omega})(-g_n)=-h_n=0$ the element $(q_\Lambda\otimes id)(-g_n-d\hat f_n)$ has a $d$-potential in $H^n(\mathfrak g)\otimes C^\infty(M)$. Let $p\in (\Lambda^n\mathfrak g^*)_{cl}\otimes C^\infty(M)$ be a $\delta$-representative of such a potential. We set $f_n:=\hat f_n + p$. As $p$ is $\delta$-closed, we have $\delta f_n=\delta \hat f_n=-g_{n+1}$. It follows
\begin{align*}
(q_\Lambda\otimes id)(-g_n-df_n)=(q_\Lambda\otimes id)(-g_n-d\hat f_n) - (q_\Lambda\otimes id)(id\otimes d)p&\\=(q_\Lambda\otimes id)(-g_n-d\hat f_n)-(q_\Lambda\otimes id)(-g_n-d\hat f_n)&=0\in H^n(\mathfrak g)\otimes \Omega^1_{cl}(M),
\end{align*}
implying that there exists a $\delta$-potential $\hat f_{n-1}$ of $(-g_n-df_n)$. Putting $f_{n-1}:=\hat f_{n-1}$ we have found two maps $f_n$ and $f_{n-1}$ satisfying $\delta f_n=-g_{n+1}$ and $\delta f_{n-1} + df_{n}=-g_n$  and hence proven:
\begin{Lemma}\label{mslemma}
Let $(M,\omega)$ be an n-plectic manifold. If $h_{n+1}=0$ and $h_{n}=0$, then there exists a multi-moment map in the sense of  \textcolor{black}{Madsen and Swann}.
\end{Lemma}

\begin{Remark}
If further classes $h_i$ are zero, we can continue the above procedure instead of setting $f_{n-1}=\hat f_{n-1}$ and construct further components of the co-moment map. If all $h_i$ are zero this gives an iterative procedure to construct a homotopy co-moment map. 
\end{Remark}

\noindent A direct consequence of Lemma \ref{mslemma} is the existence statement of  \textcolor{black}{Theorem 3.14 of \cite{MR3079342}}.
\begin{Corollary}
Let $(M,\omega)$ be n-plectic. If $H^n(\mathfrak g)$ and $H^{n+1}(\mathfrak g)$ are zero then there exists a multi-moment map.
\end{Corollary}

\bibliography{mybib}{}
\bibliographystyle{plain}

\end{document}